\documentclass[journal]{IEEEtran}

\usepackage{xcolor}
\usepackage{pdfsync}
\usepackage{graphicx}
\usepackage{epsfig}
\usepackage{amsmath}
\usepackage{amssymb}
\usepackage{amsthm}
\usepackage{mathtools}
\usepackage{amsfonts}
\usepackage{lineno}
\usepackage{tikz}
\usepackage{dblfloatfix}
\usetikzlibrary{positioning,shapes}
\usetikzlibrary{arrows}
\newtheorem{theorem}{Theorem}
\newtheorem{definition}[theorem]{Definition}
\newtheorem{corollary}[theorem]{Corollary}
\newtheorem{lemma}[theorem]{Lemma}
\newtheorem{assumption}[theorem]{Assumption}

\newtheorem{remark}[theorem]{Remark}

\newtheorem{example}[theorem]{Example}

\newcommand{\norm}[1]{\left\Vert #1\right\Vert}

\newcommand{\abs}[1]{\left|#1\right|}

\def\field#1{\mathbb #1}%
\def\R{\field{R}}%
\def\Z{\field{Z}}%
\newcommand{\Rn}[1][n]{\R^{#1}}

\newcommand{\Rp}{\R_{\geq 0}}
\newcommand{\Rsp}{\R_{> 0}}
\newcommand{\Zp}{\Z_{\geq 0}}
\newcommand{\Zsp}{\field{N}}

\def\K{\mathcal{K}}%
\DeclareMathOperator{\id}{id}
\def\linf{\ell^{\infty}}
\def\KL{\mathcal{KL}}%
\def\Kinf{\mathcal{K}_\infty}%
\let\ol=\overline%
\let\ul=\underline%
\def\A{\mathcal{A}}%

\usepackage{float}

\usepackage{color}

\newcounter{CounterEquation} 

\title{Non-conservative discrete-time ISS small-gain conditions for closed sets}

\author{%
  Navid~Noroozi,
  Roman Geiselhart\thanks{Roman Geiselhart is with the %
  University of Ulm, %
  Institute of Measurement, Control and Microtechnology, %
  Albert-Einstein-Allee 41, %
  89081 Ulm, Germany, \texttt{roman.geiselhart@uni-ulm.de}}%
,
  Lars~Gr\"une\thanks{Lars Gr\"une is with the %
    University of Bayreuth, %
    Mathematical Institute, %
    Universit\"atsstra\ss e 30, %
    95440 Bayreuth, Germany, \texttt{lars.gruene@uni-bayreuth.de}}%
,
  Bj\"orn~S.~R\"uffer\thanks{Bj\"orn S. R\"uffer is with the %
    University of Newcastle (UON), %
    Faculty of Science, 
    University Drive, %
    Callaghan, NSW 2308, %
    Australia, \texttt{bjorn.ruffer@newcastle.edu.au}},
  Fabian~R.~Wirth\thanks{Navid Noroozi and Fabian R. Wirth are with the %
    University of Passau, %
    Faculty of Computer Science and Mathematics, %
    Innstra\ss e 33, %
    94032 Passau, %
    Germany, \texttt{navidnoroozi@gmail.com}, \texttt{fabian.lastname@uni-passau.de}. The work of N. Noroozi was supported by the Alexander von Humboldt Foundation.}
}

\begin{document}
\maketitle

\begin{abstract}
This paper presents a unification and a generalization of the small-gain theory subsuming a wide range of existing small-gain theorems.
In particular, we introduce small-gain conditions that are necessary and sufficient to ensure input-to-state stability (ISS) with respect to \emph{closed} sets.
Toward this end, we first develop a Lyapunov characterization of $\omega$ISS via finite-step $\omega$ISS Lyapunov functions.
Then, we provide the small-gain conditions to guarantee $\omega$ISS of a network of systems. 
Finally, applications of our results to partial input-to-state stability, ISS of time-varying systems, synchronization problems, incremental stability, and distributed observers are given.
\end{abstract}

\begin{IEEEkeywords}
large-scale discrete-time systems, Lyapunov methods, input-to-state stability
\end{IEEEkeywords}

\maketitle

\section{Introduction}

There have been many contributions to the stability analysis of large-scale systems over the last few decades.
However, it is still challenging to analyze the stability of interconnected systems with nonlinearities, and it is desirable to develop stability conditions which can be applied to a wide range of large-scale nonlinear systems.
Among various tools toward this end, Lyapunov-based small-gain theory has received considerable attention over the last few decades; e.g. \cite{Jiang.1996,Jiang.2001,Laila.2003,Dashkovskiy.2010,Liu.2012,Liu.2014,Sanfelice.2014} to name just a few.
These small-gain conditions essentially rely on the notion of input-to-state stability (ISS) \cite{Sontag.1995} and/or related notions such as input-to-output stability \cite{Sontag.1999} and input-output-to-state stability \cite{Jiang.2005}.
The main idea is to consider a large-scale system split into smaller subsystems and analyze each subsystem individually.
In that way, it is assumed that the other subsystems act as perturbations.
Then, if the influence of the subsystems on each other is small enough, stability of the original system can be concluded.
In practice, such a treatment is often conservative as each subsystem has to be individually ISS.
This raises the question: What if we let subsystems have a stabilizing effect on each other rather than simply considering them as a perturbation to each other?
In that way, one expects to see subsystems that are unstable when decoupled from the other subsystems.
This paper endeavors to provide a contribution toward development of small-gain conditions that ensure stability of \emph{discrete-time} interconnected systems whose subsystems do not individually have to meet the same stability property, as considered in~\cite{Gielen.2015,Geiselhart.2015,Geiselhart.2016,Noroozi.2014}.

We provide small-gain conditions, referred to as non-con\-ser\-vative small-gain conditions, which are necessary and sufficient to ensure ISS with respect to closed (\emph{not} necessarily compact) sets.
The non-conservative small-gain conditions rely on the existence of a Lyapunov-like function called a \emph{finite-step Lyapunov function}, which was originally introduced by Aeyels and Peuteman~\cite{Aeyels.1998} and termed as this in \cite{Geiselhart.2014c}.
Such a function is not required to satisfy a dissipation inequality at each time step.
Instead, it only needs to satisfy a dissipation-like inequality after some finite (but constant) time.
Gielen and Lazar~\cite{Gielen.2015} studied global exponential stability (GES) for a feedback interconnection of discrete-time systems and developed small-gain conditions which are necessary and sufficient to assure GES for the interconnected system.
The small-gain conditions in \cite{Gielen.2015} were extended to the case of global asymptotic stability (GAS) by Geiselhart \emph{et al.}~\cite{Geiselhart.2015}.
More recently, Geiselhart and Wirth~\cite{Geiselhart.2016} developed non-conservative small-gain conditions, but the necessity of the conditions is \emph{only} guaranteed for networks that are exponentially ISS (i.e. solutions of the unperturbed system are decaying exponentially).
As a special case of our current results, we present non-conservative small-gain conditions for ISS (not necessarily exponential ISS) of networks.

The small-gain conditions developed in this paper unify and generalize a wide range of existing small-gain theorems.
Toward the unification and the generalization, we first provide Lyapunov characterizations of the so-called $\omega$ISS property (i.e., ISS with respect to closed sets).
The introduction of Lyapunov characterizations of $\omega$ISS is nontrivial, in general.
Most works in the literature have focused on ISS with respect to compact sets, e.g., \cite{Sontag.1996,Jiang.2001}.
Recently, progress toward Lyapunov characterizations of $\omega$ISS has been reported in \cite{Kellett.2012,Tran.2015b,Tran.2015a}.
In particular, it is shown in~\cite{Tran.2015b} that the existence of a dissipative-form $\omega$ISS Lyapunov function implies $\omega$ISS.
To show the converse, the authors, however, impose a \emph{compactness condition}, i.e. the converse ISS Lyapunov function is only obtained with respect to compact sets.
The authors also establish that the implication-form is equivalent to the dissipative-form under this compactness condition.
Motivated by~\cite{Grune.2014}, it is shown in~\cite{Tran.2015a} that a \emph{strong} variant of the implication-from ISS Lyapunov function implies $\omega$ISS.
However, as in \cite{Tran.2015b}, the compactness condition is imposed to obtain the converse.
Unlike \cite{Tran.2015b} and~\cite{Tran.2015a}, we establish Lyapunov function characterizations of ISS with respect to closed (not necessarily compact) sets.
In particular, our first result shows an equivalence between seemingly different Lyapunov characterizations of ISS.
More precisely, we introduce a so-called max-form $\omega$ISS Lyapunov function that is shown to be equivalent to the two other forms of Lyapunov functions: the implication-form and the dissipative-form $\omega$ISS Lyapunov functions.
The max-form $\omega$ISS Lyapunov function is particularly relevant for large-scale systems, which are our motivation for the introduction of this form of Lyapunov functions.
We also show that ISS with respect to a closed set is equivalent to the existence of an ISS Lyapunov function.
It should be noted that our results on the notion of $\omega$ISS are novel not only for finite-step Lyapunov functions but also for classic Lyapunov functions.
Then, we use the results to develop non-conservative small-gain conditions.
In summary, the contribution in this paper is as follows: three different Lyapunov characterizations of $\omega$ISS are given (see Theorem~\ref{thm:equivalence-of-Lyapunov-functions} below); necessary and sufficient conditions for $\omega$ISS of a large-scale system based on estimates of subsystems solutions are provided (see Theorems~\ref{thm:small-gain} and~\ref{thm:reverse-small-gain} below); these results cover the existing small-gain conditions, including those in \cite{Jiang.2001,Laila.2003,Liu.2012,Gielen.2015,Geiselhart.2015,Geiselhart.2016,Noroozi.2014}, as a special case.
These would not have been done without establishing the equivalence between the max-form $\omega$ISS Lyapunov function and the two other forms of Lyapunov functions. Further discussion about the challenges overcome in the paper is given later (see the paragraph immediately after Theorem~\ref{thm:reverse-small-gain}).
Finally, we illustrate the flexibility of our approach by reformulating several engineering and scientific problems including partial ISS, ISS for time-varying systems, incremental stability, and distributed observers as $\omega$ISS of networks.

This paper is organized as follows:
First relevant notation is recalled in Section~\ref{sec:notation}.
Then Lyapunov characterizations of $\omega$ISS for discrete-time systems via finite-step $\omega$ISS Lyapunov functions are introduced in Section~\ref{sec:iss-lyapunov-functions}.
Small-gain conditions for $\omega$ISS of large-scale discrete-time systems are provided in Section~\ref{sec:iss-networks-of-systems}.
Applications of our results are given in Section~\ref{sec:applications}.
Section~\ref{sec:conclusions} concludes the paper.

\section{Notation} \label{sec:notation}

In this paper, $\Rp (\Rsp)$ and $\Zp (\Zsp)$ denote the nonnegative
(positive) real numbers and the nonnegative (positive) integers,
respectively.
The $i$th component of $v \in \Rn$ is denoted by
$v_i$.
For any $v, w \in \Rn$, we write $v \gg w$ ($v \geq w$) if and only if $v_i > w_i$ ($v_i \geq w_i$) for each $i \in \{1,\dots,n\}$.
Let $\{e_i\}_{i=1}^{n}$ be the standard basis of $\Rn$.
For any $x \in \Rn$, $x^\top$ denotes its transpose.
We write $(x,y)$ to represent $[x^\top,y^\top]^\top$ for $x \in \Rn,y \in \R^p$.

A norm $\mu$ on $\Rn$ is called \emph{monotonic} if $0\leq x\leq y$
implies $\mu(x)\leq\mu(y)$.
In particular, $\abs{x}$ and $\abs{x}_\infty$, respectively, denote the Euclidean norm and the maximum norm for $x \in \Rn$.
Given a nonempty set $\mathcal{A}\subset\Rn$ and any point $x\in\Rn$, we denote
$\abs{x}_\A \coloneqq \inf\limits_{y \in \A}\abs{x-y}$.
Let a nonempty compact set $\A \subset \Rn$ be given.
A function $\omega \colon \Rn \to \Rp$ is said to be a \emph{proper indicator} for $\A$ if $\omega$ is continuous, $\omega(x)=0$ if and only if $x\in\mathcal{A}$, and $\omega(x) \to \infty$ when $\abs{x} \to \infty$.

Given a function $\varphi \colon \Zp \to \Rn[m]$, its sup-norm (possibly infinite) is denoted by $\norm{\varphi} = \sup \{ \abs{\varphi(k)} \colon k \in \Zp \} \leq \infty$.
The set of all functions $\Zp \to \Rn[m]$ with finite sup-norm is denoted by $\linf$.

A function $\rho \colon \Rp \to \Rp$ is positive definite if it is continuous, zero at zero and positive otherwise.
A positive definite function $\alpha$ is of class-$\K$ ($\alpha \in \K$) if it is strictly increasing. 
It is of class-$\Kinf$ ($\alpha \in \Kinf$) if $\alpha \in \K$ and also $\alpha(s) \to \infty$ if $s \to \infty$.
A continuous function $\beta \colon \Rp \times \Rp \to \Rp$ is of class-$\mathcal{KL}$ ($\beta \in \mathcal{KL}$), if for each $s \geq 0$, $\beta(\cdot,s)\in\mathcal{K}$, and for each $r \geq 0$, $\beta (r,\cdot)$ is non-increasing with $\beta (r,s)\to 0$ as $s \to \infty$.
The interested reader is referred to \cite{Kellett.2014} for more details about comparison functions.

The identity function is denoted by $\id$. Composition of functions is
denoted by the symbol $\circ$ and repeated composition of, e.g., a
function $\gamma$ by $\gamma^{i}$.
For positive definite functions $\alpha,\gamma$ we write $\alpha<\gamma$ if $\alpha(s)<\gamma(s)$ for all $s>0$.

\section{$\omega$ISS Lyapunov Functions}\label{sec:iss-lyapunov-functions}

This section provides a characterization of $\omega$ISS for discrete-time systems via finite-step $\omega$ISS Lyapunov functions.
Consider the discrete-time system
\begin{equation}%
  \label{eq:1}
  \Sigma\colon \qquad x (k+1) = g (x(k),u(k))
\end{equation}
with state $x (k) \in \Rn$ and inputs or controls $u \colon \Zp \to \Rn[m]$, $u \in \linf{}$. 
We assume that $g \colon \Rn \times \Rn[m] \to \Rn$ is continuous. 
For any initial value $\xi \in \Rn$ and any input $u \in \linf{}$, $x(\cdot,\xi,u)$ denotes the corresponding solution to \eqref{eq:1}.

\begin{definition}
A continuous and positive semi-definite function $\omega \colon \Rn \to \Rp$ is called a \emph{measurement function}.
\end{definition}

\begin{definition}\label{def:gKb}
Given a measurement function $\omega$, we call the function $g$ in~\eqref{eq:1} \emph{globally $\K$-bounded} with respect to $\omega$ if there exist $\kappa_1,\kappa_2 \in \Kinf$ such that
\begin{align}    \label{eq:2}
    \omega(g(\xi,\mu)) \leq \kappa_1 (\omega(\xi)) + \kappa_2 (\abs{\mu}) 
  \end{align}
  for all $\xi \in \Rn$ and all $\mu \in \Rn[m]$.
\end{definition}
Note that this definition extends the global $\K$-boundedness definition from~\cite{Geiselhart.2016} by including a measurement function. 
The importance of the global $\Kinf$-boundedness will become obvious in the subsequent results.
Further discussion is provided in the last paragraph of this section.

We borrow the notion of input-to-state stability with respect to a
measurement function $\omega$ from \cite{Tran.2015b} and \cite{Tran.2015a}.

\begin{definition} 
The discrete-time system \eqref{eq:1} is \emph{input-to-state stable} with respect to a measurement function $\omega$ ($\omega$ISS) if there exist $\beta \in \mathcal{KL}$ and $\gamma \in \K$ such that for all $\xi \in \Rn$, all $u \in \linf{}$ and all $k \in \Zp$ we have
\begin{equation} \label{eq:ISSestimate}
\omega\big(x(k,\xi,u)\big) \leq \max \Big\{ \beta\big(\omega(\xi),k\big), \gamma(\norm{u}) \Big\}.
\end{equation}
\end{definition}

\begin{remark}%
If $\omega$ is given by a norm, i.e. $\omega(\cdot) = \abs{\cdot}$, then we recover the notion of input-to-state stability (ISS)~\cite{Sontag.1989,Jiang.2001}.
In that case, continuity of $g$ in \eqref{eq:1} implies global $\K$-boundedness, as shown in \cite{Geiselhart.2017}.
Also, $\omega$ISS subsumes the notion of state independent-input-to-output stability (SI-IOS) \cite{Sontag.1999,Jiang.2005} by taking $\omega (\cdot) = \abs{h(\cdot)}$ where $h \colon \Rn \to \Rn[m]$ is the continuous output mapping of system \eqref{eq:1}.
Moreover, set-stability versions are also covered by allowing the measurement function $\omega$ to be defined as the distance to a (closed) set.
Similarly, (robust) stability of a prescribed motion can be considered by suitable choice of $\omega$ (see \cite{El-Hawwary2013,Skjetne2002} for instance). As shown in \cite{Kellett.2012}, it also includes a weak form of incremental ISS \cite{Angeli.2002} where the input-dependent bound is the essential supremum of inputs of the systems rather than the difference between the two inputs.
\end{remark}

Now we introduce several characterizations of $\omega$ISS Lyapunov
functions.

\begin{definition} \label{D:lyap-func}
  Let $\omega$ be a measurement function and $M\in\Zsp$. Let
  $V \colon \Rn \to \Rp$ be continuous and let there exist 
  $\ul\alpha,\ol\alpha\in\Kinf$ so that for all $\xi\in\Rn$,
  \begin{gather}
    \ul\alpha (\omega(\xi)) \leq V(\xi) \leq \ol\alpha (\omega(\xi))  \label{eq:3}.
  \end{gather}
  The function $V$ is called
  \begin{itemize}
  \item {a} \emph{max-form finite-step $\omega$ISS Lyapunov function}
    for~\eqref{eq:1} if there exist $\alpha_{\text{max}} \in \Kinf$ with
    $\alpha_{\text{max}} < \id$ and $\gamma_{\text{max}} \in \K$ such that for all $\xi \in \Rn$
    and all $u \in \linf{}$,
    \begin{gather}
      V(x(M,\xi,u)) \leq \max \{ \alpha_{\text{max}}(V(\xi)), \gamma_{\text{max}} (\norm{u}) \}.  \label{eq:4}
    \end{gather}
    
  \item {an} \emph{implication-form finite-step $\omega$ISS Lyapunov
      function} for~\eqref{eq:1} if there exists a positive definite
    function $\alpha_{\text{imp}}$ and a function $\gamma_{\text{imp}} \in \K$ such that for all
    $\xi \in \Rn$ and all $u \in \linf$,
    \begin{gather}
        \begin{aligned}
          V(\xi) \geq \gamma_{\text{imp}}(\norm{u}) \qquad \implies\qquad \\
          V(x(M,\xi,u)) - V(\xi) \leq -\alpha_{\text{imp}}(V(\xi)) .
        \end{aligned}
      \label{eq:5}
    \end{gather}
  \item {a} \emph{dissipative-form finite-step $\omega$ISS Lyapunov
      function} for~\eqref{eq:1} if there exist $\alpha_{\text{diss}} \in \Kinf$
    with $\alpha_{\text{diss}} < \id$ and $\gamma_{\text{diss}} \in \K$ such that for all
    $\xi \in \Rn$ and all $u \in \linf$,
    \begin{gather}
      V(x(M,\xi,u)) - V(\xi) \leq - \alpha_{\text{diss}}(V(\xi)) + \gamma_{\text{diss}}(\norm{u}) . \label{eq:6}
    \end{gather}
  \end{itemize}
  For the case $M=1$ we drop the term ``finite-step'' and instead
  speak of a ``classic'' $\omega$ISS Lyapunov function of the
  respective type. 
\end{definition}

Whenever we say ``$\omega$ISS Lyapunov function'' without further
qualification, we refer to any of the three types defined above, or
the type is determined by the context where this reference appears.
This inaccuracy is motivated by the fact that all these three forms are equivalent provided that a global $\K$-boundedness property holds, as Theorem~\ref{thm:equivalence-of-Lyapunov-functions} below shows.
Note that, because of the $\K$-boundedness property we do not have to assume that $\omega$ is a proper indicator function, as done in~\cite{Jiang.2001,Tran.2015b,Tran.2015a}.

\begin{remark} \label{R:01}
  \begin{itemize}
  \item
    We point out that an equivalent form of \eqref{eq:5}, which is given in \cite{Jiang.2001,Tran.2015b}, is the following
    \begin{align*}
      \omega(\xi) \geq \tilde\gamma_{\text{imp}}(\norm{u}) \qquad \implies\qquad \\
      V(x(M,\xi,u)) - V(\xi) \leq - \alpha_{\text{imp}}(V(\xi)),
    \end{align*}
    where $\tilde\gamma_{\text{imp}} (\cdot):= \ul\alpha^{-1} \circ \gamma_{\text{imp}} (\cdot)$.
    Here we place emphasis on \eqref{eq:5} as it simplifies exposition of proofs.
  \item
    Via a rescaling of the Lyapunov function $V$, i.e., $W\coloneqq \rho(V)$ for some suitable $\Kinf$ function $\rho$, there is no loss of generality in assuming that the functions $\alpha_{\text{imp}}$ and $\gamma_{\text{imp}}$ in \eqref{eq:5} are of class $\Kinf$, see \cite[Remark 3.3]{Jiang.2005} and \cite[Remark 3.3]{Jiang.2001} for details.
  \end{itemize}
\end{remark}

\begin{theorem}
\label{thm:equivalence-of-Lyapunov-functions} Let $g\colon\Rn \times \Rn[m] \to \Rn$ be continuous. 
  Then the following properties are equivalent.
  \begin{enumerate}
  \item System \eqref{eq:1} admits a max-form finite-step
    $\omega$ISS Lyapunov function. 
  \item System \eqref{eq:1} admits an implication-form finite-step $\omega$ISS Lyapunov function and $g$ is globally $\K$-bounded.
  \item System \eqref{eq:1} admits a dissipative-form finite-step $\omega$ISS Lyapunov function.
  \item System~\eqref{eq:1} is $\omega$ISS.
\end{enumerate}
In particular, the constant $M \in \Zsp$, which by Definition~\ref{D:lyap-func} appears in the $\omega$ISS Lyapunov functions, can be chosen arbitrarily in the items $(i)$, $(ii)$ and $(iii)$.
\end{theorem}

\begin{proof}
  We prove the equivalences by proving the implications $(i) \Rightarrow (iii) \Rightarrow (iv) \Rightarrow (ii) \Rightarrow (i)$.

  \paragraph{The implication $(i) \Rightarrow (iii)$}
Let $V$ be a max-form finite-step $\omega$ISS Lyapunov function satisfying~\eqref{eq:3} and~\eqref{eq:4}.
To prove the implication, we will apply the construction from the proof of~\cite[Lemma~2.8]{Jiang.2002} to rescale $V$ by a suitable function $\rho\in\Kinf$ in order to get a dissipative-form finite-step $\omega$ISS Lyapunov function  $W: \R^n\to \R_+$, $W(\xi):=\rho(V(\xi))$ satisfying~\eqref{eq:6}.\\ 
  First, if $\alpha_{max}(V(\xi))\geq \gamma_{max}(\norm{u})$ then~\eqref{eq:4} implies $V(x(M,\xi,u))\leq\alpha_{max}(V(\xi))$.
  Then the proof of~\cite[Lemma~2.8]{Jiang.2002} yields that for any $\hat\alpha \in \Kinf$ satisfying $\hat\alpha(s) \leq (\id-\alpha_{max}) (s)$ if $s\leq 1$ and $\hat\alpha(s) \leq s/2$ if $s>1$, we have
\begin{equation}\label{eq:pf-RG1}
	W(x(M,\xi,u)) -W(\xi) \leq -\hat \alpha (V(\xi))= -(\hat \alpha\circ \rho^{-1}) (W(\xi)). 
\end{equation}
On the other hand, if $\alpha_{max}(V(\xi))\leq\gamma_{max}(\norm{u})$ then~\eqref{eq:4} implies 
\begin{equation}\label{eq:pf-RG2}
	W(x(M,\xi,u)) = \rho(V(x(M,\xi,u))) < (\rho\circ\gamma_{max})(\norm{u}).
\end{equation}
Finally, by~\eqref{eq:pf-RG1} and~\eqref{eq:pf-RG2}, we have
\begin{align*}%
	W(x(M,\xi,u)) -W(\xi) & \\
	&\hspace{-2cm} \leq -\min\{\hat \alpha\circ \rho^{-1},\id\} (W(\xi)) + (\rho\circ\gamma_{max})(\norm{u}). 
\end{align*}
Thus, $W$ satisfies~\eqref{eq:6} with $\alpha_{diss}:=\frac12 \min\{\hat \alpha\circ \rho^{-1},\id\}\in \Kinf$ and $\gamma_{diss}:=\rho\circ\gamma_{max}$, which proves that $W$ is a dissipative-form finite-step $\omega$ISS Lyapunov function.

\paragraph{The implication $(iii) \Rightarrow (iv)$}
Let $V$ be a dissipative-form finite-step $\omega$ISS Lyapunov function satisfying~\eqref{eq:6}, which is equivalent to
\begin{equation*}
V(x(M,\xi,u)) \leq \rho(V(\xi)) + \gamma_{diss}(\norm{u})
\end{equation*}
with positive definite $\rho$ satisfying $(\id-\rho)\in \Kinf$,
see~\cite[Remark~3.7]{Geiselhart.2016}.
The proof follows~\cite[Theorem~4.1]{Geiselhart.2016}, mutatis mutandis, replacing norms by the measurement function $\omega$.

\paragraph{The implication $(iv) \Rightarrow (ii)$}
It is obvious that any $\omega$ISS Lyapunov function is a finite-step $\omega$ISS Lyapunov function with $M=1$.
The implication $(iv) \Rightarrow (ii)$ follows from \cite[Theorem 6]{Tran.2015b}, where it is shown that $\omega$ISS implies the existence of an implication-form $\omega$ISS Lyapunov function.
Moreover, it follows from the fact that system~\eqref{eq:1} is $\omega$ISS that $g$ is globally $\K$-bounded.

\paragraph{The implication $(ii) \Rightarrow (i)$}
Let $V$ be an implication-form finite-step $\omega$ISS Lyapunov function satisfying the implication~\eqref{eq:5} for some $M\in \Zsp$ and functions $\gamma_{imp}$ and $\alpha_{imp}$.
We will show that $V$ is a max-form finite-step $\omega$ISS Lyapunov function.
To do this, we have to find suitable functions $\alpha_{max}$ and $\gamma_{max}$ satisfying~\eqref{eq:4}.
  First, by applying~\eqref{eq:3}, we get
  \begin{equation*}
  V(x(M,\xi,u))  \leq \ol\alpha ( \omega(x(M,\xi,u)))
  \end{equation*}
for all $\xi \in \R^n$, $u(\cdot) \in \linf$. Making use of Lemma~\ref{lem:evolutionKbound} in the appendix, which requires the global $\K$-boundedness property, we see that there exist $\K$-functions $\vartheta_M$ and $\zeta_M$ satisfying
\begin{equation*}
\omega(x(M,\xi, u)) \leq \vartheta_M(\omega(\xi)) + \zeta_M(\norm{u})
\end{equation*}
for all $\xi \in \R^n$, $u(\cdot) \in \ell^\infty$.
 Hence, by applying~\eqref{eq:3}, we get for all $\xi \in \R^n$ with $V(\xi)< \gamma_{imp}(\norm{u})$ 
\begin{align} 
V(x(M,\xi,u)) &< \ol\alpha \big( \vartheta_M(\ul\alpha^{-1}\circ \gamma_{imp} (\norm{u})) + \zeta_M(\norm{u}) \big) \nonumber\\
&=: \gamma_{max}(\norm{u}). \label{eq:pf.1}
 \end{align}
On the other hand, for all $\xi \in \R^n$ with $V(\xi) \geq \gamma_{imp}(\norm{u})$, inequality~\eqref{eq:5} directly yields
\begin{equation*}
V(x(M,\xi,u)) \leq (\id-\alpha_{imp}) (V(\xi)) =: \alpha_{max}(V(\xi)).
\end{equation*}
Note that we can without loss of generality assume that $\alpha_{max} \in \Kinf$ by considering $\alpha_{imp}$ suitably small. Then $V$ is max-form finite-step $\omega$ISS Lyapunov function satisfying~\eqref{eq:4}.
\end{proof}

\begin{remark}
It is obvious that every classic $\omega$ISS Lyapunov function is a finite-step $\omega$ISS Lyapunov function with $M=1$.
On the other hand, a finite-step $\omega$ISS Lyapunov function for system~\eqref{eq:1} is a classic $\omega$ISS Lyapunov function for a system that is constructed as follows.
The evaluation of the solutions of~\eqref{eq:1} at the times $j M$, $j \in \Zp$ can be described by a dynamic equation.
In that way, we obtain
  \begin{align}%
    \label{eq:7}
    y (k+1) = f (y(k),w(k)) 
  \end{align}
where $y (k) \in \Rn$, $w(k) \coloneqq \big(u(k),u(k+1),\dots, u(k+M-1)\big)$ and, with the notation $g^{k+1}(x,u_{1},\ldots,u_{k+1}) \coloneqq g\big(g^{k}(x,u_{1},\ldots,u_{k}),u_{k+1}\big)$ for $k\geq 1$, and $g^{1}(x,u)\coloneqq g(x,u)$, we define
\begin{align*}
f (y,w) \coloneqq g^{M}(y,u_{1},\ldots,u_{M}) .
\end{align*}
With the same arguments as those in~\cite[Remark~4.2]{Geiselhart.2016}, we can show that a function $V \colon \Rn \to \Rp$ is a finite-step $\omega$ISS Lyapunov function for system~\eqref{eq:1} if and only if it is a classic $\omega$ISS Lyapunov function for system~\eqref{eq:7}.
\end{remark}

We note that our result uses the global $\K$-boundedness, where similar related results, e.g., \cite{Jiang.2001,Tran.2015b,Tran.2015a}, instead require that $\omega$ is a proper indicator function.
This, in turn, implies that the existing works only provide Lyapunov characterizations of ISS with respect to compact sets.
We do not have to assume that $\omega$ is a proper indicator anywhere in this work.
Moreover, Theorem~\ref{thm:equivalence-of-Lyapunov-functions} implies the equivalence between max-form and dissipative-form finite-step $\omega$ISS Lyapunov functions, whereas implication-form finite-step $\omega$ISS Lyapunov functions requires an additional condition.
While the proper indicator property in~\cite{Tran.2015b,Tran.2015a} is only a sufficient condition to prove $\omega$ISS, we can directly conclude from Theorem~\ref{thm:equivalence-of-Lyapunov-functions} that the global $\K$-boundedness is sufficient \emph{and} necessary. 
This comes from the fact that every ISS system necessarily satisfies the global $\K$-boundedness condition.
In particular, assume that~\eqref{eq:1} is $\omega$ISS, then we have from~\eqref{eq:ISSestimate} that
$$
\omega\big(g(\xi,u)\big) = \omega\big(x(1,\xi,u)\big) \leq \max \Big\{ \beta\big(\omega(\xi),1\big), \gamma(\abs{u}) \Big\},
$$
which implies global $\K$-boundedness with $\kappa_1 (\cdot):= \beta\big(\cdot,1\big)$ and $\kappa_2 (\cdot) := \gamma(\cdot)$.

While Theorem~\ref{thm:equivalence-of-Lyapunov-functions} presents Lyapunov characterizations of ISS with respect to a \emph{single} measurement function, an extension of these results to the case of ISS with respect to two measurement functions (see~\cite[Definition 1]{Tran.2015b}) can be easily given if the measurement functions satisfy a so-called commensurability condition~\cite[Assumption 2]{Tran.2015b}.

Figure~\ref{fig:omegaISSequivalences} summarizes our contributions in this section over the existing literature.
\begin{figure*}[!t]
\normalsize
  \centering
  \colorlet{known}{black}
  \colorlet{obvious}{black}
  \colorlet{ours}{black}
  \small
  \begin{tikzpicture}[x=1.4in,y=0.9in,prop/.append style={rectangle,draw=black}, shorten <=2pt, shorten >=2pt]
    \node[prop] (df-fs-LF) at (0,2) {$\exists$ diss.-form f.-s.\ $\omega$ISS LF};
    \node[prop] (mf-fs-LF) at (-.85,1) {\vbox{\hbox{$\exists$ max-form f.-s.}\hbox{$\omega$ISS LF}}};
    \node[prop] (if-fs-LF-K-bdd) at (0,0) {\vbox{\hbox{$\exists$ impl.-form f.-s. $\omega$ISS LF}\hbox{\& global $\mathcal{K}$-boundedness}}};
    
    \node[prop] (df-LF) at (2,2) {$\exists$ diss.-form $\omega$ISS LF};
    \node[prop] (omegaISS) at (0.8,1) {$\omega$ISS};
    \node[prop] (if-LF) at (2,0) {$\exists$ impl.-form $\omega$ISS LF};
    \node[prop] (if-fs-LF) at (0,-1) {$\exists$ impl.-form f.-s.\ $\omega$ISS LF};
\node[shape=rectangle,draw=white] (our-result) at (0.1,1) {\large{Theorem~\ref{thm:equivalence-of-Lyapunov-functions}}};
%
    \draw[ours,-implies,double equal sign distance] (omegaISS) to node[below right] {} (if-fs-LF-K-bdd);
    \draw[known,-implies,double equal sign distance,
    in=45,out=-135
    ] (df-LF) to node[above] {\hspace{-1cm}\cite{Tran.2015b}} (omegaISS);
    \draw[dashed,known,-implies,double equal sign distance,bend right] (if-LF) to node[above]{\hspace{0.35cm}\cite
      {Tran.2015b}} (omegaISS);
    \draw[known,-implies,double equal sign distance,bend right] (omegaISS) to node[below] {\hspace{-.45cm}\cite
      {Tran.2015b}} (if-LF);
    \draw[dashed,known,-implies,double equal sign distance] (if-LF) to node[right] {\cite{Tran.2015b}} (df-LF);
    \draw[obvious,-implies,double equal sign distance] (df-LF) to node[above] {obvious} (df-fs-LF);
    \draw[ours,implies-,double equal sign distance] (df-fs-LF) to node[left] {} (mf-fs-LF);
    \draw[ours,-implies,double equal sign distance] (df-fs-LF) to node[left] {} (omegaISS);
    \draw[obvious,-implies,double equal sign distance] (if-fs-LF-K-bdd) to node[above] {\hspace{1.5cm}obvious} (if-fs-LF);
    \draw[obvious,-implies,double equal sign distance,in=0,out=-90] (if-LF) to node[right] {\hspace{0.5cm}obvious} (if-fs-LF.east);
    \draw[ours,implies-,double equal sign distance] (mf-fs-LF) to node[left] {} (if-fs-LF-K-bdd);
  \end{tikzpicture}

  \caption[Relations between different system properties.]{Relations
    between different system properties as shown in this paper and
    known from the literature. From \cite{Tran.2015b}, the existence of an implication-form $\omega$ISS Lyapunov function without any additional assumptions does not necessarily imply $\omega$ISS. Dashed arrows indicate implications from \cite{Tran.2015b} that require the measurement function $\omega$ to be a proper indicator. This, however, is only a sufficient and not a necessary condition. On the other hand, Theorem~\ref{thm:equivalence-of-Lyapunov-functions} shows that the global $\K$-boundedness is sufficient and necessary to conclude $\omega$ISS.}
  \label{fig:omegaISSequivalences}
\hrulefill  
\vspace*{4pt} \end{figure*}

\section{$\omega$ISS Networks of Systems} \label{sec:iss-networks-of-systems}

This section establishes $\omega$ISS for a network of discrete-time
systems. In particular, it extends finite-step small-gain conditions in
\cite{Geiselhart.2016,Noroozi.2014} to the case of $\omega$ISS. To this end,
assume that system~\eqref{eq:1} can be decomposed into $\ell$
interconnected subsystems
\begin{equation}%
  \label{eq:19}
  \Sigma_i \colon \quad x_i (k+1) = g_i (x_1(k),\dots,x_\ell(k),u(k)) ,
\end{equation}
where each
$g_i \colon \Rn[n_1] \times \dots \times\Rn[n_\ell]\times \Rn[m] \to
\mathbb{R}^{n_i}$ is continuous and $x_i (k) \in \R^{n_i}$ for each
$i \in \{ 1, \dots, \ell \}$, and inputs $u \in \linf(\Zp,\R^m)$.
Given $n=n_1+\ldots+n_\ell$, $x := ( x_1,\dots, x_\ell )$ and $g := ( g_1,\dots, g_\ell)$, we call \eqref{eq:1} the composite system of the subsystems \eqref{eq:19}.

We make the following stability assumptions on the network of systems
$\Sigma_{i}$, $i=1,\ldots,\ell$.

\begin{assumption} \label{ass:thm-small-gain}
Let measurement functions $\omega_i \colon \mathbb{R}^{n_i} \to \Rp$, $i = 1,2, \dots,\ell$
and $M\in\Zsp$ be given. Suppose that for each subsystem $\Sigma_{i}$,
there exists a continuous function $W_i \colon \R^{n_i} \to \Rp$ such that the
following hold
\begin{enumerate}
\item There exist functions
  $\overline{\alpha}_i,\underline{\alpha}_i \in \Kinf$ such that for
  all $\xi_i \in \R^{n_i}$
  \begin{equation}
    \ul\alpha_i (\omega_i (\xi_i)) \leq W_i (\xi_i) \leq \overline{\alpha}_i (\omega_i (\xi_i)) .  \label{eq:ch3e11}
  \end{equation}
\item There exist $\gamma_{ij} \in \Kinf \cup \{ 0 \}$, $j=1,\dots,\ell$, and
  $\gamma_{iu} \in \K \cup \{ 0 \}$ such that for all
  $\xi \in \Rn$ and all $u \in \linf{}$ the estimate
  \begin{equation}
    \begin{aligned}
      &W_i (x_i (M,\xi,u)) \\
      &\leq \max \bigg\{ \max_{j \in \{ 1,\dots, \ell \} } \Big\{ \gamma_{ij} (W_{j} (\xi_{j})) \Big\}, \gamma_{iu}(\norm{u}) \bigg\} 
    \end{aligned} 
    \label{eq:ch3e13}
  \end{equation}
  holds, where $x_i (\cdot,\xi,u)$ denotes the $i$th component of the solution $x(\cdot,\xi,u)$ of $\Sigma$, that corresponds to the subsystem $\Sigma_i$. 
\end{enumerate}
\end{assumption}

We emphasize that~\eqref{eq:ch3e11}--\eqref{eq:ch3e13} is not of the
form \eqref{eq:3}--\eqref{eq:4}, i.e., it is not a max-form finite-step
$\omega_{i}$ISS Lyapunov function characterization for the $i$th
subsystem. The reason is that, in principle, it is an estimate for
trajectories of the composite system~\eqref{eq:1}, or rather the $i$th
component of this trajectory. Indeed, the left hand side of estimate
\eqref{eq:ch3e13} does take into account the dynamics of neighboring
subsystems. However, depending on the sparsity of the interconnection
graph, this may require only local information and not the forward
solution of the entire composite system, cf.\ the explanation in
Fig.~\ref{fig:neighbouring-dynamics}.

\begin{figure}[htbp]
  \centering
  \tikzset{sys/.style={circle,draw,#1,thick,fill=#1!10,minimum size=.25in}}
  \begin{tikzpicture}[node distance=1.5cm and 1.5cm, every path/.append style={thick}, shorten <=2pt, shorten >=2pt]
    \node[shape=circle,draw=black, inner sep=2pt] (c) {0};
    \node[shape=diamond,draw=black,inner sep=1pt] (n-1) [left=of c] {1};
    \node[shape=rectangle,draw=black] (n-2) [left=of n-1] {2};
    \node[shape=diamond,draw=black,inner sep=1pt] (n+1) [right=of c] {1};
    \node[shape=rectangle,draw=black] (n+2a) [right=of n+1] {2};
    \node[shape=diamond,draw=black,inner sep=1pt] (n+2b) [above=of n+1] {1};
    \node[shape=rectangle,draw=black] (n+3) [right=of n+2b] {2};
    \draw[-latex] (n-2) -- (n-1);
    \draw[-latex] (n-1) -- (c);
    \draw[latex-] (c) -- (n+1);
    \draw[latex-] (n+1) -- (n+2a);
    \draw[latex-] (n+1) -- (n+2b);
    \draw[latex-] (c) -- (n+2b);
    \draw[latex-] (n+2b) -- (n+3);
  \end{tikzpicture}
  \caption{%
    With the information about initial states of direct neighbors (diamonds) and
    the local node itself (circle) at $k=0$, the state of the local node at $k=1$ can be
    computed, cf.~\eqref{eq:19}. To compute its state at $k=2$, the states
    of neighbors (diamonds) at $k=1$ are required, which in turn require the
    knowledge of states of their neighbors (squares) at $k=0$.
    \newline
    Hence, for $M=2$ a finite-step stability estimate of any of the
    forms \eqref{eq:4}, \eqref{eq:5}, or \eqref{eq:6} requires knowledge
    of the initial conditions (states at $k=0$) not only at the local node
    (circle), but also at direct neighbors (diamonds) and
    neighbors of these neighbors (squares), along with the states
    at the direct neighbors (squares) at $k=1$.
    \newline
    In addition, exogenous input values at $k=0$ and $k=1$ are
    required at the local node (circle) and at $k=0$ at its direct
    neighbors (squares).%
    \newline
    Note that for these computations only the edges directed toward
    the local node need to be considered.
  }
  \label{fig:neighbouring-dynamics}
\end{figure}

\begin{theorem}%
  \label{thm:small-gain} 
  Consider the systems~\eqref{eq:19} and let Assumption~\ref{ass:thm-small-gain} hold.
  Assume the functions
  $\gamma_{ij}$ given in~\eqref{eq:ch3e13} satisfy
  \begin{equation}%
    \gamma_{i_1 i_2}\circ
    \gamma_{i_{2}i_{3}}\circ
    \dots
    \circ
    \gamma_{i_{r-1}i_{r}}\circ
    \gamma_{i_{r}i_{1}}<\id
    \label{eq:41}
  \end{equation}
  for all sequences $(i_1,\ldots,i_r)\in\{1,\ldots,\ell\}^{r}$ and
  $r=1,\ldots,\ell$.
  Let $\mu\colon\Rn\to\Rp$ be any monotonic norm.
  Then with
  $$
  \omega(x)\coloneqq
  \mu\big(\omega_1(x_1),\ldots,\omega_{\ell}(x_{\ell})\big)
  $$
  the
  composite system $\Sigma$ is $\omega$ISS.
\end{theorem}
We also note that there are various equivalent formulations of the
so-called \emph{small-gain condition}~\eqref{eq:41}, cf.\
\cite{Dashkovskiy.2010,Dashkovskiy.2007,Ruffer.2010}, and it is the same condition imposed in the classic small-gain theory for interconnections of stable systems.

In contrast to the classic small-gain theory, cf., e.g.,
\cite{Dashkovskiy.2010,Dashkovskiy.2007}, the above small-gain theorem does not
require each individual subsystem to be stable when considered in
isolation. In fact, some subsystems may well be open-loop unstable, as
long as in the local interconnection with other subsystems they
satisfy estimates~\eqref{eq:ch3e11} and \eqref{eq:ch3e13}. In other words,
this result allows for neighboring subsystems to exercise a
``stabilizing effect'' on a given local system. See~\cite{Geiselhart.2015,Geiselhart.2016,Noroozi.2014} for more discussions and examples.

The proof of Theorem~\ref{thm:small-gain} follows the lines of the
classic small-gain result in~\cite{Dashkovskiy.2010}.
The main difference is the use of the finite-step estimate~\eqref{eq:ch3e13} instead of ISS Lyapunov functions.

\begin{proof}[Proof of Theorem~\ref{thm:small-gain}]
We prove this result by explicitly constructing a max-form finite-step $\omega$ISS Lyapunov function.
As the proof follows the construction in~\cite{Dashkovskiy.2010}, we only give a sketch of it.

As shown in~\cite[Theorem~5.5]{Ruffer.2010}, the small-gain condition~\eqref{eq:41} implies the existence of $\Kinf$-functions $\sigma_i$, $i=1, \ldots, \ell$, satisfying
\begin{equation}\label{eq:sigma_estimate}
\max_{j \in \{1, \ldots, \ell\}} \sigma_i^{-1} \circ \gamma_{ij} \circ \sigma_j < \id
\end{equation}
for each $i=1, \ldots, \ell$.
With~\eqref{eq:sigma_estimate}, we have $\alpha:= \max_{i,j \in \{1, \ldots, \ell\}} \sigma_i^{-1} \circ \gamma_{ij} \circ \sigma_j < \id$.
As in~\cite{Dashkovskiy.2010} we define $V \colon \Rn \to \Rp$ by
\begin{equation*}
V (\xi) \coloneqq \max_i \sigma_i^{-1} (W_i(\xi_i))
\end{equation*}
for all $\xi=(\xi_1,\ldots,\xi_\ell) \in \Rn[n]$. It follows with~\eqref{eq:ch3e11} and the definition of $\omega$,
its monotonicity, and the equivalence of norms on $\Rn$ that there exist $\ul\alpha,\ol\alpha \in \Kinf$ such that for all $\xi\in\Rn$ we have
\begin{equation}\label{eq:prop&posdef}
\ul\alpha (\omega(\xi)) \leq V (\xi) \leq \overline{\alpha} (\omega(\xi)). 
  \end{equation}
Moreover, we get
  \begin{align*}
& \hspace{-1.2cm} V (x(M,\xi,u))
    =  \max_i \sigma_{i}^{-1} \big(W_i(x_i(M,\xi,u))\big) \\
     \qquad \stackrel{ \eqref{eq:ch3e13}}{\leq} & \max_{i,j} \sigma_i^{-1}\!\Big(\!\max\big\{\gamma_{ij}\big(W_j(\xi_j)\big),\gamma_{iu}(\norm{u})\big\}\!\Big)\\
     \qquad= &\max_{i,j}  \sigma_i^{-1} \!\Big(\! \max \big\{ \gamma_{ij} \circ \sigma_j \circ \sigma_j^{-1} \big(W_j (\xi_j)\big) , \gamma_{iu} (\norm{u}) \!\big\}\! \Big) \\
     \leq & \max_{i,j,l} \sigma_{i}^{-1} \!\Big(\! \max \big\{ \gamma_{ij} \circ \sigma_j \circ \sigma_l^{-1} (W_l(\xi_l)) , \gamma_{iu} (\norm{u}) \!\big\} \!\Big)  \\
    \leq  &\max \Big\{ \!\alpha\big( \!V (\xi) \big), \max_i (\sigma_i^{-1} \circ \gamma_{iu}) (\norm{u}) \!\Big\}. 
  \end{align*}
This shows that $V$ is a max-form finite-step $\omega$ISS Lyapunov function for system~\eqref{eq:1}, which is thus $\omega$ISS by Theorem~\ref{thm:equivalence-of-Lyapunov-functions}.
\end{proof}

The strength of the previous theorem becomes all the more apparent,
when we note that in fact any $\omega$ISS system can essentially be decomposed
into subsystems satisfying stability estimates like~\eqref{eq:ch3e11}--\eqref{eq:ch3e13}.
The precise formulation is as follows.

\begin{theorem} \label{thm:reverse-small-gain}
Consider a system of the form~\eqref{eq:1}.
Suppose that~\eqref{eq:1} is $\omega$ISS with a max-form $\omega$ISS Lyapunov function $V$ satisfying~\eqref{eq:3} for $\ul \alpha,\ol\alpha \in \Kinf$ and~\eqref{eq:4} for $M=1$, $\alpha_{\text{max}} \in \Kinf$, $\alpha_{\text{max}} < \id$ and $\gamma_{\text{max}} \in \K$.
In addition, assume that there exists some integer $\hat M \geq 1$, such that for all $s \in \Rsp$ it holds that
\begin{equation} \label{eq:24}
\alpha_{max}^{\hat M} (s) < \ul{\alpha} \circ \Big( \frac{1}{c} \id \Big) \circ \overline{\alpha}^{-1} (s)
\end{equation}
where $c \geq 1$ is some constant.

Assume that $\omega$ can be decomposed as $\omega(x) = \mu\big(\omega_1(x_1),\ldots,\omega_\ell(x_\ell)\big)$, where $x=(x_1,\dots,x_\ell)$, $x_i \in \R^{n_i}$, $\omega_{i}$ are measurement functions on $\R^{n_i}$ and $\mu$ is a monotonic norm on $\Rn$.
If
\begin{align} \label{eq:12}
& \mu(z) \leq c \abs{z}_\infty \qquad\quad \forall z \in \R^\ell , &
\end{align}
then there exist continuous functions $W_i \colon \R^{n_i} \to \Rp$,
$\gamma_{ij} \in \Kinf \cup \{ 0 \}$, and
$\gamma_{iu} \in \mathcal{K} \cup \{ 0 \}$, $i,j = 1,\dots,\ell$
satisfying \eqref{eq:ch3e11}, \eqref{eq:ch3e13} and~\eqref{eq:41} with the same integer $\hat M$ as in~\eqref{eq:24}.
\end{theorem}

We shall see in the proof that, in fact, it is sufficient to assume that the decomposition of $\omega$ satisfies the lower bound $\omega(x) \geq \mu\big(\omega_1(x_1),\ldots,\omega_\ell(x_\ell)\big)$.

\begin{proof}
By Theorem~\ref{thm:equivalence-of-Lyapunov-functions}, note that the existence of $\omega$ISS Lyapunov function $V$ is equivalent to $\omega$ISS.
For simplicity of notation, let $\alpha := \alpha_{max}$ and $\gamma := \gamma_{max}$.

It follows from~\eqref{eq:4} and the fact that $\alpha < \id$ that
\begin{equation} \label{eq:25}
V(x(k,\xi,u)) \leq \max \{ \alpha^k ( V (\xi) ), \gamma (\norm{u}) \} . 
\end{equation}
Denote $\mu_{i}\coloneqq \mu(e_{i})$, the norm of the $i$th standard basis vector.
Define $W_i \colon \R^{n_i} \to \Rp$ by $W_i (\xi_i) \coloneqq \mu_i \omega_i (\xi_i)$ for all $i = 1,\dots, \ell$.
Obviously, this choice of $W_i(\cdot)$ satisfies~\eqref{eq:ch3e11} with $\ul\alpha_i = \ol\alpha_i = \id$.

Let $\hat M \in \Zsp$ be given by~\eqref{eq:24}.
Then
\begin{align*}
W_i \big(x_i (\hat M,\xi,u) \big)
& = \mu_{i} \omega_i \big(x_i (\hat M,\xi,u)\big)  \\
& = \mu \Big(\omega_i \big(x_i (\hat M,\xi,u)\big)e_i \Big) \\ %
&\leq \mu \Big(\omega_{1}\big(x_{1} (\hat M,\xi,u)\big),\ldots,\omega_{\ell} \big(x_\ell (\hat M,\xi,u) \big) \Big)\\
    &= \omega\big(x (\hat M,\xi,u)\big) \leq \ul\alpha^{-1} \big( V (x(\hat M,\xi,u))\big)
\end{align*}
with $\ul\alpha$ from~\eqref{eq:3}.
From this inequality and \eqref{eq:25} it follows that
\begin{align*}
    W_i \big(x_i (\hat M,\xi,u) \big)
    & \leq \underline{\alpha}^{-1} \Big( V \big(x(\hat M,\xi,u)\big)\Big)\\
    & \leq \underline{\alpha}^{-1}  \Big(\max \Big\{ \alpha^{\hat M} \big( V (\xi) \big), \gamma (\norm{u}) \Big\}\Big)  \\
    & \leq \underline{\alpha}^{-1} \Big (\max \Big\{ \alpha^{\hat M} \circ \overline{\alpha} \big(\omega(\xi)\big),
      \gamma (\norm{u})\Big\}\Big)  \\
    & = \max \Big\{ \big(\ul\alpha^{-1} \circ \alpha^{\hat M} \circ \overline{\alpha}\big) \big(\omega(\xi)\big), \\
& \qquad\qquad \big(\ul\alpha^{-1} \circ \gamma\big) (\norm{u}) \Big\} 
  \end{align*}
with $\ol\alpha$ from~\eqref{eq:3}.
By~\eqref{eq:12}, we have
\begin{align*}
 W_i \big(x_i (\hat M,\xi,u) \big)
\! \leq &\!\max \Big\{ \!\! \max_j \big(\ul\alpha^{-1} \! \circ \! \alpha^{\hat M} \circ  \overline{\alpha} \! \circ \! (c \id) \big) \big(\omega_j (\xi_j) \big), \\
& \qquad\quad \big(\underline{\alpha}^{-1} \circ \gamma\big) (\norm{u}) \Big\} .
\end{align*}
Denote $\chi \coloneqq \underline{\alpha}^{-1} \circ \alpha^{\hat M} \circ \overline{\alpha} \circ (c \id)$ and observe that it follows from our initial assumption in~\eqref{eq:24} that $\chi (s) < s$ for all $s > 0$.
With this and the definition of $W_{i}$, our last estimate becomes
\begin{equation}  \label{eq:20}
W_i \big(x_i (\hat M,\xi,u) \big) \leq \max \Big\{  \max_j \chi \big( W_j (\xi_j) \big), \big(\underline{\alpha}^{-1} \circ \gamma\big) (\norm{u}) \Big\} . 
\end{equation}
Now define $\gamma_{ij} \coloneqq \chi$ for all $i,j=1, \dots, \ell$ and $\tilde\gamma \coloneqq \underline{\alpha}^{-1} \circ \gamma$.
With these we obtain from~\eqref{eq:20} that
$$
W_i (x_i(\hat M,\xi,\,u)) \leq \max \{ \max_j \gamma_{ij} ( W_j(\xi_j) ),
  \tilde\gamma (\norm{u}) \}
$$
and so \eqref{eq:ch3e13} holds.
The fact that $\gamma_{ij} < \id$ for all $i,j=1, \dots, \ell$ ensures that condition~\eqref{eq:41} holds as well.
\end{proof}

Theorem~\ref{thm:small-gain} provides sufficient conditions for $\omega$ISS of the overall system from estimates of solutions of the subsystems of the interconnected system.
It is reasonable to ask how conservative these conditions are.
Theorem~\ref{thm:reverse-small-gain} shows the necessity of these conditions if the gain functions associated with the Lyapunov function $V$ satisfy inequality~\eqref{eq:24} (see Example~\ref{ex:justification-ex} below where this inequality is verified for an illustrative example).
Unlike~\cite{Geiselhart.2016}, where mainly dissipative-form ISS Lyapunov functions are used and, because of that, the necessity of small-gain conditions is only shown for networks that admit the global $\K$-boundedness assumption with a linear $\kappa_1$ in~\eqref{eq:2}, no limitation on growth rate of the function $\kappa_1$ is required in our work.
Our result is obtained by using a max-form ISS Lyapunov function and the equivalence of this form of Lyapunov functions with the two other forms provided by Theorem~\ref{thm:equivalence-of-Lyapunov-functions}.

In the proof of Theorem~\ref{thm:reverse-small-gain}, the choice of taking $W_i(\xi_i):= \mu_i \omega(\xi_i)$ extends the ideas of \cite{Geiselhart.2014c,Geiselhart.2015,Geiselhart.2016}, where the authors show that any scaled norm satisfies a particular finite-step decay condition under a suitable condition such as~\eqref{eq:24}.
Whereas many former Lyapunov-based small-gain results require the knowledge of given ISS Lyapunov functions for each subsystem, Theorem~\ref{thm:reverse-small-gain} proposes a \emph{systematic} approach to satisfy the conditions in Theorem~\ref{thm:small-gain} by setting $W_i(\xi):= \mu_i \omega_i(\xi_i)$ and iterating the dynamics until \eqref{eq:ch3e13} and~\eqref{eq:41} hold.

Note that, in particular, this approach extends the non-conservative small-gain results given in \cite{Geiselhart.2015,Geiselhart.2016} to the case of ISS with respect to a measurement function $\omega$. 
In~\cite{Geiselhart.2016}, the authors state a sufficient small-gain result that was shown to be also necessary for the class of \emph{exponentially input-to-state stable} systems, which are ISS systems satisfying~\eqref{eq:ISSestimate} with $\beta(s,t)=C\tau^ts$ with $C\geq1$ and $\tau \in (0,1)$.
This class of systems is included in Theorem~\ref{thm:reverse-small-gain} as any exponentially ISS system satisfies~\eqref{eq:24}, see~\cite[Theorem~IV.8(i)]{Geiselhart.2015}.
However, we emphasize that Theorem~\ref{thm:reverse-small-gain} is applicable to a larger class than the class of exponentially ISS systems.
This observation is made explicit in the following example.

\begin{example} \label{ex:justification-ex}
Inspired by~\cite[Example~16]{Geiselhart.2014c}, we consider the discrete-time system
\begin{equation} \label{eq:example-system}
x(k+1) = g(x(k),u(k)), \quad k \in \Zsp,
\end{equation}
with state $x(k) \in \R^2$, input $u(k) \in \R$ and continuous dynamics $g:\R^2\times \R \to \R^2$ defined by
\begin{equation} \label{eq:example-dynamics}
(\xi,\mu) \mapsto \left( \begin{array}{c} \max\{ \xi_2-\xi_2^2, \frac12 \xi_2, \mu \} \\ \max\{ \xi_1-\xi_1^2, \frac12 \xi_1, \mu \}\end{array}
\right).
\end{equation}
As the measurement function $\omega$, we consider the sup-norm $\omega(\xi):=|\xi|_\infty =\max \{|\xi_1|,|\xi_2|\}$, hence, in this case $\omega$ISS reduces to ISS.
Then we get, for all $\xi \in \R^2$ and all $\mu \in \R$,
\begin{align*}
\omega(g(\xi,\mu)) &= \max \{ |g_1(\xi,\mu)|, |g_2(\xi,\mu)| \} \\
 &\leq \max \{ \max\{\omega(\xi)-\omega^2(\xi), \frac12 \omega(\xi)\}, |\mu| \} \\
 & \leq \kappa_1(\omega(\xi)) + \kappa_2(|\mu|)
\end{align*}
with $\kappa_1(s):= \max\{s-s^2,\frac12s\}$ and $\kappa_2(s):=s$. This shows that $g$ defined in~\eqref{eq:example-dynamics} is globally $\K$-bounded.

Next, we will show that system~\eqref{eq:example-system} is ISS by applying Theorem~\ref{thm:small-gain}.
To this end, we consider system~\eqref{eq:example-system} decomposed into the subsystems
\begin{equation*}
x_i(k+1) = g_i(x(k),u(k)), \quad k \in \Zsp,
\end{equation*}
where $g_i$ denotes the $i$th component of $g$.
Let $W_i(\xi_i):= |\xi_i|$ for $i\in \{1,2\}$. Then, for all $\xi \in \R^2$ and all $\mu \in \R$, we have
\begin{align*}
W_1(g_1(\xi,\mu)) & = \big|\max\{ \xi_2-\xi_2^2, \tfrac12 \xi_2, \mu \}\big| \\
& \leq \max\{ |\xi_2|-\xi_2^2, \tfrac12 |\xi_2|, |\mu| \},
\end{align*}
which yields the estimate~\eqref{eq:ch3e13} with $M=1$ and gain functions $\gamma_{11} \equiv 0$, $\gamma_{12}(s) = \max\{s-s^2,\frac12s\}$ and $\gamma_{1u}(s)=s$. Similarly,~\eqref{eq:ch3e13} holds for the second subsystem with $M=1$ and gain functions $\gamma_{21}(s) = \max\{s-s^2,\frac12s\}$, $\gamma_{22} \equiv 0$ and $\gamma_{2u}(s)=s$.
As $\gamma_{12} \circ \gamma_{21} < \id$, the small-gain condition~\eqref{eq:41} holds, and Theorem~\ref{thm:small-gain} implies that system~\eqref{eq:example-system} is ISS.

It is not surprising that we could satisfy the conditions of Theorem~\ref{thm:small-gain} with the choice $W_i(\xi_i)= |\xi_i|$. This comes from the fact that also the conditions of Theorem~\ref{thm:reverse-small-gain} are satisfied. To see this, consider the function $V \colon \R^2 \to \Rp$ defined by $V(\xi):= \omega(\xi)$, which obviously satisfies~\eqref{eq:prop&posdef} with $\underline \alpha \equiv \overline \alpha \equiv \id$.
Moreover, it is easy to see that
\begin{align*}
V(g(\xi,\mu)) & = \max \{ |g_1(\xi,\mu)|, |g_2(\xi,\mu)| \} \\
& \leq \max \{ \alpha_{max} (V(\xi)), \gamma_{max}(|\mu|) \}
\end{align*}
for all $\xi \in \R^2$ and all $\mu \in \R$ with $\alpha_{max}(s):= \max \{ s-s^2, \frac12s \}$ and $\gamma_{max}(s):=s$ for all $s \in \Rp$. Clearly, $\alpha_{max}<\id$. Hence, $V$ is a max-form ISS Lyapunov function.
Observe that we can set $c=1$ in~\eqref{eq:12}, which implies that~\eqref{eq:24} holds for all $s>0$ and all $M\geq 1$.
As all conditions in Theorem~\ref{thm:reverse-small-gain} are satisfied, the choice of $W_i(\xi_i)= |\xi_i|$ has to satisfy the conditions in Theorem~\ref{thm:small-gain} for a suitably large $M\in\Zsp$, which, in this example, can even be taken as $M=1$.

Finally, we will show that system~\eqref{eq:example-system} is not exponentially ISS, i.e., there do not exist $C\geq 1$, $\tau \in (0,1)$ satisfying
\begin{equation*}
\omega(x(k,\xi,u(\cdot))) \leq \max\{C \tau^k \omega(\xi), \gamma_{max}(\|u\|)\}
\end{equation*}
for all $k\in \Zsp$, $\xi \in \R^2$ and $u \in \ell^\infty$.
To see this, let $u(\cdot) \equiv 0$ and observe that for all $k\in \Zsp$ and all $\xi \in \Rp^2$ it holds
\begin{equation*}
\omega(x(k,\xi,0)) = \omega(g^k(\xi,0)) = \alpha_{max}^k(\omega(\xi)).
\end{equation*}
 As
 \begin{align*}
 \lim_{k \to \infty} \frac{\omega(x(k+1,\xi,0))}{\omega(x(k,\xi,0))} & = \lim_{k \to \infty} \frac{\alpha_{max}(\omega(x(k,\xi,0)))}{\omega(x(k,\xi,0))} \\
& = 1 - \lim_{k \to \infty} \omega(x(k,\xi,0)) =1,
 \end{align*}
i.e., the decay rate of any solution approaches $1$, the system
\begin{equation*}
x(k+1) = g(x(k),0), \quad k \in \Zsp
\end{equation*}
cannot have a globally exponentially stable origin, and thus, system~\eqref{eq:example-system} cannot be exponentially ISS.
\end{example}

\section{Applications} \label{sec:applications}

We verify the versatility of our results by reformulating several engineering and scientific problems including partial stability theory, ISS for time-varying systems, synchronization problems, incremental stability and distributed observers as $\omega$ISS of networks.

\subsection{Partial ISS for interconnected systems}

Consider the following nonlinear system
\begin{subequations}%
  \label{eq:28}
  \begin{align}  
    & x_{1} (k+1) = g_{1} (x_1(k),x_2(k),u(k)) & \label{eq:29} \\
    & x_{2} (k+1) =  g_{2} (x_1(k),x_2(k),u(k)) & \label{eq:30}
  \end{align}
\end{subequations}
with the state
$x_1 (k) \in \Rn[n_{1}], x_{2} (k) \in \Rn[n_2]$ and the input $u (k) \in \R^m$.
We refer to $x_1$ and $x_2$ as the \emph{primary} variable and the \emph{auxiliary} variable, respectively.
Similarly, we also respectively call \eqref{eq:29} and \eqref{eq:30} the primary subsystem and the auxiliary subsystem.
Denote $x \coloneqq (x_1,x_2) \in \Rn$ with $n \coloneqq n_1 + n_2$.
Also, let $g_i \colon \Rn[n_1] \times \Rn[n_2] \times \Rn[m] \to \Rn[n_i]$ for $i = 1,2$ be continuous and $g_1 (0,x_2,0) = 0$ for all $x_2 \in \Rn[n_2]$.

Roughly speaking, by the notion of partial ISS we mean ISS of the primary
subsystem~\eqref{eq:29} with respect to the input $u$ while the auxiliary
subsystem~\eqref{eq:30} is not required to be stable.
Here we precisely define the notion we refer to as partial ISS.

\begin{definition} \label{def:partial-iss}
System \eqref{eq:28} is uniformly input-to-$x_1$-state stable if there exist $\beta \in \mathcal{KL}$ and $\gamma \in \K$ such that for all $\xi \in \Rn$, all $u \in \linf{}$ and all $k \in \Zp$ we have
\begin{equation*}
\abs{x_1(k,\xi,u)} \leq \max \Big\{ \beta\big(\big|\xi_1\big|,k \big), \gamma(\norm{u}) \Big\}.
\end{equation*}
\end{definition}
In other words, uniform input-to-$x_1$-state stability is $\omega$ISS with respect to the measurement function $\omega(x_1,x_2) = |x_1|$.

When $g \coloneqq (g_1, g_2)$ in~\eqref{eq:28} is globally $\K$-bounded, Theorem~\ref{thm:equivalence-of-Lyapunov-functions} ensures that the uniform input-to-$x_1$-state stability is equivalent to the existence of a continuous function $V \colon \Rn[n_1] \to \Rp$, some integer $M > 0$, functions $\ul\alpha,\ol\alpha \in \Kinf$, $\gamma \in \K$, and $\alpha \in \Kinf$ with $\alpha < \id$ such that for all $\xi_1 \in \Rn[n_1]$, $\xi_2 \in \Rn[n_2]$ and all $u \in \linf{}$,
\begin{gather*}
    \ul\alpha (|\xi_1|) \leq V (\xi_1,\xi_2) \leq \ol\alpha (|\xi_1|) \\
    V (x_1(M,\xi,u),x_2(M,\xi,u)) \leq \max \{ \alpha(V (\xi_1,\xi_2)), \gamma (\norm{u}) \}.
\end{gather*}
Obviously, the zero level set of $V$ is not compact, but closed.

\begin{remark}
While this paper concentrates on time-invariant systems, our results can be applied to time-varying systems by transforming the time-varying system into a time-invariant one of the form \eqref{eq:28}. To see this, consider the following time-varying system
\begin{equation}  \label{eq:31}
x (k+1) = g (k,x(k),u(k))
\end{equation}
where $x(k) \in \Rn$, $u(k) \in \R^m$ and $g \colon \Z \times \Rn \times \Rn[m] \to \Rn$ is continuous.
Let us transform \eqref{eq:31} into
\begin{subequations}   \label{eq:32}
\begin{align}  
      & x(k+1) = g (z(k),x(k),u(k)) &\\
      & z(k+1) = z(k) + 1 . &
\end{align}
\end{subequations}
In terms of partial ISS, system \eqref{eq:32} is of the form \eqref{eq:28} with the primary variable $x_1 = x$ and the auxiliary variable $x_2 \coloneqq z$; thus Theorem~\ref{thm:equivalence-of-Lyapunov-functions} can be applied with a right choice of measurement function $\omega$ with respect to $x$.  This gives the discrete-time counterpart of results in \cite{Chellaboina.2002} where relations between partial stability and stability theory for time-varying \emph{continuous-time} systems have been addressed.
\end{remark}

Now assume that system~\eqref{eq:28} can be decomposed into $\ell$ interconnected subsystems as follows
\begin{subequations}%
  \label{eq:33}
  \begin{align}  
    x_{1i} (k+1) & = g_{1i} (x_{11}(k),\dots,x_{1\ell}(k),x_{21}(k),\dots, x_{2\ell}(k),u(k))\\
    x_{2i} (k+1) & = g_{2i} (x_{11}(k),\dots,x_{1\ell}(k),x_{21}(k),\dots, x_{2\ell}(k),u(k))
  \end{align}
\end{subequations}
with the state $x_{1i} (k) \in \Rn[n_{1i}], x_{2i} (k) \in
\Rn[n_{2i}]$ and the input $u (k) \in \R^m$.
Denote $x_i \coloneqq (x_{i1},\dots,x_{i\ell}) \in
\Rn[n_i]$ for $i = 1,2$ with
$n_i \coloneqq \sum_{j=1}^\ell n_{ij}$.
Also, let $g_{1i} \colon \Rn[n_1]\times \Rn[n_2] \times \R^{m} \to
\Rn[n_{1i}]$ and
$g_{2i} \colon \Rn[n_1]\times \Rn[n_2] \times \R^{m} \to \Rn[n_{2i}]$ be continuous with $g_{i1} (0,x_{2},0) = 0$ for
all $x_{2} \in \Rn[n_{2}]$.
Denote $x_i \coloneqq (x_{1i},x_{2i})$ and $g_i \coloneqq (g_{1i},g_{2i})$.
To enable the stability analysis of the composite system, we make the following assumption.

\begin{assumption}  \label{A:02}
Suppose that for each subsystem \eqref{eq:33} there exist $W_i \colon \Rn[n_{1_i}] \times \R^{n_{2_i}} \to \Rp$, $\ul\alpha_i ,\ol\alpha_i \in \Kinf$, $\gamma_{ij} \in \Kinf \cup \{ 0 \}$, $\gamma_{u i} \in \mathcal{K} \cup \{ 0 \}$ and $M>0$ such that for all $(\xi_1,\xi_2) \in \Rn[n_1] \times \Rn[n_2]$ and all $u \in \linf{}$ the following hold
\begin{align*}
& \ul\alpha_i (|\xi_{1i}|) \leq W_i (\xi_{1i},\xi_{2i}) \leq \ol\alpha_i (| \xi_{1i}|) , & 
    \\
    &\begin{aligned}
      & W_i (x_{1i} (M,\xi_1,\xi_2,u), x_{2i} (M,\xi_1,\xi_2,u)) \\
      &\leq \max \bigg\{ \max_{j \in \{ 1,\dots, \ell \} } \Big\{\gamma_{ij} (W_{j} (\xi_{1j},\xi_{2j})) \Big\}, \gamma_{u i} (\norm{u}) \bigg\}
  \end{aligned}
  \end{align*}
and assume that the functions $\gamma_{ij}$ also satisfy \eqref{eq:41}. 
\end{assumption}

It follows from Theorem~\ref{thm:small-gain} that under Assumption~\ref{A:02},
system \eqref{eq:33} is uniformly input-to-$x_1$-state stable. This is summarized by the
following corollary.

\begin{corollary}%
  \label{C:01}
  Let Assumption~\ref{A:02} hold. Also, let $g_i$ in \eqref{eq:33} for
  $i=1,\dots,\ell$ be globally $\K$-bounded. Then system \eqref{eq:33} is uniformly input-to-$x_1$-state stable.
\end{corollary}

As an application of Corollary~\ref{C:01}, we give an example of
synchronization of oscillator networks. Consider the dynamics of a
network of oscillators described by
\begin{equation}%
  \label{eq:36}
  z_i(k+1) = f_i(z_i) + \sum_{j \in \mathcal{N}_i} a_{ij} \Upsilon \psi_{ij} (z_i,z_j) , \qquad i = 1,\dots,\ell 
\end{equation}
where the state $z_i \in \Rn[p]$ and $\mathcal{N}_i$ denotes the
neighbors of node $i$.
The continuous function
$f_i \colon \Rn[p] \to \Rn[p]$ represents the dynamics of each
uncoupled node, $\Upsilon \in \Rn[p\times p]$ is a constant matrix
describing the type of the coupling, and $a_{ij} \in \R$ with
$a_{ij} = a_{ji}$ represents the weighting of coupling on each link of
a network. 
The coupling functions $\psi_{ij} \colon \Rn[p] \times\Rn[n_1] \to \Rn[n_1]$ are
continuous with $\psi_{ij} (x,y) = - \psi_{ji} (y,x)$ for all
$(x,y) \in \Rn[p] \times \Rn[p]$ and determine the law of interaction
between node $i$ and node $j$.
Synchronization is typically analyzed by finding a reference trajectory $r(k)$ and determining if $z_i (k) \to r(k)$.
A common reference trajectory is the system average $\ol z = \frac{1}{\ell}\sum_{i=1}^\ell z_i$.
In this case, each state $z_i$ is compared to the system average \cite{Olfati-Saber.04}, using $e_i = z_i - \ol z$. If each $e_i \to 0$ then the system synchronizes.
In that way, the average and error dynamics are given by
\begin{subequations}%
  \label{eq:37}
  \begin{align}
    e_i (k+1) & = f_i(e_i+\ol z) + \sum_{j \in \mathcal{N}_i} a_{ij} \Upsilon \psi_{ij} (\ol z + e_i,\ol z + e_j) \nonumber \\
              & \qquad + \frac{1}{\ell}\sum_{j=1}^\ell f_j(\ol z + e_j)   \nonumber\\
              & \eqqcolon g_{1i} (e_1,\dots,e_\ell,\ol z) , \qquad i = 1,\dots,\ell & \\
    \ol z(k+1) & = \frac{1}{\ell} \sum_{j=1}^\ell f_j (\ol z + e_j) \eqqcolon  g_2(e_1,\dots,e_\ell,\ol z) .  &
  \end{align}
\end{subequations}
In terms of partial ISS of interconnected systems, the average $\ol z$
acts as the auxiliary variable $x_2$ (i.e.
$x_2 = \ol z$), the error state $e_i$ is the primary variable
corresponding to node $i$ (i.e. $x_{1i} = e_i$), and
$u \equiv 0$. Given $x_{21} = \ol z$ for instance, system
\eqref{eq:37} is of the form \eqref{eq:33}; thus Corollary \ref{C:01}
can be used to guarantee the synchronization of the oscillator network
\eqref{eq:36}.

\subsection{Incrementally stable interconnected systems}
Consider dynamical systems of the following form
\begin{align} \label{eq:inc-sys}
  & x(k+1) = g\big(x(k)\big) , &
\end{align}
where $x(k)\in\R^n$ is the state and $g \colon \Rn \to \Rn$ is continuous.
For any initial value $\xi \in \Rn$, $x(\cdot,\xi)$ denotes the corresponding solution to \eqref{eq:inc-sys}.
We give a definition of asymptotic incremental stability for discrete-time systems, which is borrowed from~\cite{Tran.2016a, tranrufferkellett-convergence-properties-for-discrete-time-nonlinear-systems}.
\begin{definition} \label{def:inc-stability}
  System~\eqref{eq:inc-sys} is globally incrementally asymptotically stable if there exists
  $\beta \in \KL$ such that
  \begin{equation*} 
    \abs{x (k,\xi) - x (k,\zeta)} \leq \beta\big(\abs{\xi - \zeta},k\big)
  \end{equation*}
  holds for all $\xi,\zeta \in \Rn$ and $k \in \Zp$.
\end{definition}
As in~\cite{Angeli.2002}, the notion of incremental stability can be reformulated into asymptotic stability with respect to a measurement function in the following manner. Associated with~\eqref{eq:inc-sys} is the augmented system
\begin{equation} \label{eq:inc-aux-sys}
  \begin{array}{rcl}
    z_{1} (k+1) & = & g \big( z_{1} (k) \big) \\
    z_{2} (k+1) & = & g \big( z_{2} (k) \big) 
  \end{array}
\end{equation}
where~\eqref{eq:inc-aux-sys} is formed by two copies of the original system~\eqref{eq:inc-sys}.
Given the diagonal set $\Delta \coloneqq \{ (x,x) : x \in \Rn \}$, define the measurement function $\omega$ by $\omega(\xi,\zeta) \coloneqq \abs{\xi-\zeta}$.

In~\cite[Lemma 1]{Angeli.2002} it is shown that $\frac{1}{\sqrt{2}} \abs{\xi-\zeta}$ is the Euclidean distance of the point $(\xi,\zeta)$ to $\Delta$. Hence, it follows that global asymptotic stability of~\eqref{eq:inc-aux-sys} with respect to the measurement function $\omega$ is equivalent to incremental stability of~\eqref{eq:inc-sys}.  Thus, a Lyapunov characterization for incremental stability of~\eqref{eq:inc-sys} is provided via Theorem~\ref{thm:equivalence-of-Lyapunov-functions}.  The following corollary recovers parts of Theorem~5 in~\cite{Tran.2016a}.

\begin{corollary}~\label{cor:inc-sta}
  System~\eqref{eq:inc-sys} is incrementally stable if and only if there exist a continuous function $V \colon \Rn \times \Rn \to \Rp$, $M > 0$ and functions $\ul\alpha,\ol\alpha \in \Kinf$ with $\alpha < \id$ such that for all $\xi , \zeta \in \Rn$
  \begin{align*}
    \ul\alpha \big(|\xi - \zeta|\big) \leq V (\xi,\zeta) \leq \ol\alpha \big(|\xi - \zeta|\big);
  \end{align*}
  and each solution to~\eqref{eq:inc-aux-sys} satisfies
  \begin{gather*}
    V \big(z_{1}(M,\xi), z_{2}(M,\zeta)\big) \leq \alpha\big(V (\xi,\zeta)\big) 
  \end{gather*}
  for all $\xi , \zeta \in \Rn$.
\end{corollary}
Now let us split system~\eqref{eq:inc-sys} into $\ell$ subsystems 
\begin{equation}%
  \label{eq:inc-int-sys}
  x_{i} (k+1)  = g_{i} \big(x_{1}(k),\dots, x_{\ell} (k) \big)
\end{equation}
where $x_{i} (k) \in \Rn[n_i]$,
$x = (x_{1},\dots,x_{\ell})$, $g = (g_{1},\dots,g_{\ell})$.

Theorem~\ref{thm:small-gain} allows us to conclude incremental stability of system~\eqref{eq:inc-sys} if it is an interconnection of systems~\eqref{eq:inc-int-sys}.
\begin{corollary}
  Assume that there exists an $M>0$ and for each subsystem~\eqref{eq:inc-int-sys} there exists a continuous function $W_i \colon \Rn[n_{i}] \times \R^{n_{i}} \to \Rp$, $\ul\alpha_i ,\ol\alpha_i \in \Kinf$, $\gamma_{ij} \in \Kinf \cup \{ 0 \}$
  satisfying~\eqref{eq:41}
  such that for all $\xi_i,\zeta_i \in \Rn[n_i]$,
  \begin{gather*}
    \ul\alpha_i \big(|\xi_i - \zeta_i|\big) \leq W_i (\xi_i,\zeta_i) \leq \ol\alpha_i \big(|\xi_i - \zeta_i |\big)  \\
    \intertext{and}
    W_i \big(x_i(M,\xi), x_i(M,\zeta)\big)  \leq \max_{j \in \{ 1,\dots, \ell \} } \gamma_{ij} \big(W_j (\xi_j,\zeta_j) \big).
  \end{gather*}
  Then system~\eqref{eq:inc-sys} is globally incrementally asymptotically stable.  
\end{corollary}

\subsection{Distributed observers}
\label{sec:distr-observ}

We consider the problem of constructing distributed observers for
networks of interconnected control systems. For simplicity, we will
exclude external inputs to the network from our considerations, and we
will also focus on the network interconnection aspect, rather than
discussing the construction of individual local observers (the
interested reader is referred to \cite{Shim.2016,Mazenc.2014,Andrieu.2009,Alessandri.2004,Jiang.2001b} for observer theory).

Our basic assumption is that in a network context, we have local observers of local subsystems.
We assume that the states of these local observers asymptotically converge to the true state of each subsystem, given perfect information of the true states of neighboring subsystems.

Of course such information will be unavailable in practice, and instead each local observer will at best have the state estimates produced by other, neighboring observers available for its operation.

We model this as follows using the concept of $\omega$ISS of networks introduced earlier.

\subsubsection{The distributed system to be observed}
\label{sec:distr-syst-be}

The distributed nominal system consists of $\ell$ interconnected subsystems
\begin{equation}
  \label{eq:27}
  \Sigma_{i}\colon\left\{
  \begin{aligned}
    x_{i}^{+} & = f_{i}(x_{i}, x_{j}\colon j\in \mathcal{N}_{i})\\
    y_{i} & = h_{i}(x_{i}, x_{j}\colon j\in \mathcal{N}_{i})\\    
  \end{aligned}
  \right.,\quad i=1,\ldots,\ell,
\end{equation}
where $\mathcal{N}_{i}\subset\{1,\ldots,\ell\}$ denotes the set of
neighboring subsystems to subsystem $\Sigma_{i}$ which influence directly the
dynamics of subsystem $\Sigma_{i}$. While $x_i \in \R^{n_i}$ is the state of system
$\Sigma_i$, the quantity $y_i \in \R^{p_i}$ (for some $p_i \in \Zsp$) is the output that can be measured
locally and serve as input for a state observer. As usual, we assume
$f_{i}\colon\R^{n_{i}}\times\prod_{j\in\mathcal{N}_{i}}\R^{n_{j}}\to\R^{n_i}$ and $h_i\colon\R^{n_{i}}\times\prod_{j\in\mathcal{N}_{i}}\R^{n_{j}}\to\R^{p_i}$
are both continuous. The fact that states (instead of outputs) of neighboring
systems are inputs to system $\Sigma_{i}$ is no restriction, as the
respective output maps can be absorbed into $f_{i}$.

\begin{figure*}[b]
\hrulefill  
\begin{align}\tag{\ref{counter:Eq}} 
\label{eq:rec-sys-non-norm}
\begin{array}{rcl}
\!x_1 (k+1) \!\!\!\!&\!\! = \!\!&\!\! \!\!\left( 1+ \sqrt{x_1(k)^2+x_2(k)^2} + \sin \big( \tan^{-1}\big(\frac{x_2(k)}{x_1(k)}\big) \big) \right) \! \cos \Big( \frac{1}{4} \big( 1+ \sin \big( \tan^{-1}\big(\frac{x_2(k)}{x_1(k)}\big) \big) \big) + \tan^{-1}\big(\frac{x_2(k)}{x_1(k)}\big) \Big) \\ 
\!x_2 (k+1) \!\!\!\!&\!\! = \!\!&\!\! \!\!\left( 1+ \sqrt{x_1(k)^2+x_2(k)^2} + \sin \big( \tan^{-1}\big(\frac{x_2(k)}{x_1(k)}\big) \big) \right) \! \sin \Big( \frac{1}{4} \big( 1+ \sin \big( \tan^{-1}\big(\frac{x_2(k)}{x_1(k)}\big) \big) \big) + \tan^{-1}\big(\frac{x_2(k)}{x_1(k)}\big) \Big)
\end{array}
\end{align}
\end{figure*}

\subsubsection{The structure of the distributed observers}
\label{sec:struct-distr-observ}

A local observer $\mathcal{O}_{i}$ for system $\Sigma_{i}$ would have access to $y_{i}$ and produce an estimate $\hat{x}_{i}$ of $x_{i}$ at every time step.
In order to do this and essentially to reproduce the dynamics~\eqref{eq:27}, however, it also needs to know $x_{j}$, for all $j\in\mathcal{N}_{i}$.
Access to this kind of information is very unrealistic, so instead we assume that at least it knows the estimates $\hat x_{j}$ for $j\in\mathcal{N}_{j}$ produced by neighboring observers.
This basically means our observer satisfies
\begin{equation} \label{eq:42}
\mathcal{O}_{i}\colon {\hat x_{i}}^{+} = \hat f_{i} (\hat x_{i}, y_{i}, y_{j}\colon j\in \mathcal{N}_{i} ,\hat x_{j}\colon j\in \mathcal{N}_{i})
\end{equation}
for some appropriate function $\hat f_{i}$, which we assume to be continuous.

Necessarily, the observers are coupled in the same directional sense as the original distributed subsystems.
Based on the small-gain theory introduced above, this leads us to a framework for the design of distributed observers that guarantees that interconnections of distributed observers asymptotically track the true system state.
Thus we consider the composite system given by
\begin{equation}  \label{eq:43}
\begin{array}{lcl}
    x_{i}^{+} & = & f_{i}(x_{i}, x_{j}\colon j\in \mathcal{N}_{i}), \quad
    y_{i} = h_{i}(x_{i}, x_{j}\colon j\in \mathcal{N}_{i})\\    
    {\hat x_{i}}^{+} & = & \hat f_{i} (\hat x_{i}, y_{i}, y_{j}\colon j\in \mathcal{N}_{i}, \hat x_{j}\colon j\in \mathcal{N}_{i}) .
\end{array}
\end{equation}

\subsubsection{A consistency framework for the design of distributed
 observers}

We reformulate the design of the distributed observer $\mathcal{O}=(\mathcal{O}_{1},\ldots,\mathcal{O}_{\ell})$ for the distributed system $\Sigma=(\Sigma_1,\ldots,\Sigma_\ell)$ into $\omega$ISS of networks with 

\begin{equation} \label{eq:omega-observer}
\omega (x,\hat x) \coloneqq \abs{\hat x - x}.
\end{equation}

It should be pointed out that, in view of stability with respect to sets, \eqref{eq:omega-observer} can be written as $\omega (x,\hat x) = \abs{z}_\mathcal{A}$, $\mathcal{A} = \{ z \in \R^{2n} \colon x = \hat x , z =(x,\hat x) \}$ where the set $\mathcal{A}$ is not compact, but closed. 

According to \eqref{eq:omega-observer}, we make the following assumption which implies a stability estimate for each subsystem of the composite system \eqref{eq:43}: 
There exist a continuous function $W_i \colon \R^{n_i}\to\Rp$, functions $\ol\alpha_{i},\ul\alpha_{i}\in\Kinf$, functions $\gamma_{ii} \in \Kinf\cup \{0\}$ with $\gamma_{ii} < \id$, $\gamma_{ij}\in\Kinf$ for $j\in\mathcal{N}_i$ and $M > 0$ such that for all $\xi_i,\zeta_i \in \R^{n_i}$
\refstepcounter{equation}
\begin{equation}
  \label{eq:45} 
\ul\alpha_{i}(\abs{\zeta_i - \xi_i})\leq W_i (\xi_i,\zeta_i) \leq  \ol\alpha_i (\abs{\zeta_i - \xi_i}),
\end{equation}
and for all $\xi,\zeta \in \Rn$
      \begin{align}  
 W_i & \big(x_i (M,\xi) , \hat x_i (M,\zeta,\xi)\big) \nonumber\\ 
& \leq \max \left\{ \gamma_{ii}( W_i (\xi_i,\zeta_i)),\max_{j\in\mathcal{N}_{i}}\gamma_{ij} \big( W_j (\xi_j,\zeta_j)\big)\right\}. \label{eq:46}
\end{align}
where $\hat x_i (M,\zeta,\xi)$ denotes the $i$th component of the solution to $\mathcal{O}$ with the initial value $(\zeta,\xi) \in \Rn[2n]$.

In view of \eqref{eq:ch3e11} and \eqref{eq:45}, we note that the corresponding measurement function of each subsystem of the composite system \eqref{eq:43} is given by $\omega_i (\xi_i,\zeta_i) = \abs{\zeta_i-\xi_i}$ for all $i=1,\dots,\ell$.

We pose the result of this section as a theorem, however, the proof is a direct consequence of Theorem~\ref{thm:small-gain} and thus omitted.

\begin{theorem}
  Consider a distributed system consisting of subsystems $\Sigma_i$ given by~\eqref{eq:27}, $i=1,\ldots,\ell$. Assume for each system there exists a local observer of the form~\eqref{eq:42}. Let each observer satisfy~\eqref{eq:45}. If the functions $\gamma_{ij}$ from~\eqref{eq:46} satisfy the small-gain condition~\eqref{eq:41} then the distributed observer given by $\mathcal{O}$ globally asymptotically tracks the state of the distributed system $\Sigma$ in the sense that there exists $\beta \in \KL$ so that
  $$
  \omega(x(k,\xi),\hat x(k,\zeta,\xi)) \leq \beta(\omega(\xi,\zeta),k)
  $$
for all $k \geq 0$ and any initial values $\zeta,\xi \in \Rn$.
\end{theorem}

\subsection{A non-norm measurement function}

\setcounter{CounterEquation}{\arabic{equation}}
\refstepcounter{CounterEquation} \label{counter:Eq}
\refstepcounter{equation} 

The examples given so far always used measurement functions which were defined using norms.
Here we briefly present an example where the measurement function is different. Consider the system given in~\eqref{eq:rec-sys-non-norm}, where $x_1(k),x_2 (k)\in \R$ and $\tan^{-1}(\cdot)$ is the inverse tangent function.
It can be shown that any solution to system~\eqref{eq:rec-sys-non-norm} asymptotically converges to the set $\A := \{(x_1,x_2) \in \R^2 | x_1 = 0 , x_2 \leq 0 \}$. 
Given $r := \sqrt{x_1^2 +x_2^2}$ and $\theta := \tan^{-1}(x_2/x_1)$, the system can be rewritten as
\begin{align}\label{eq:pol-sys-non-norm}
\begin{array}{rcl} 
r(k+1) &=& r(k) + 1 + \sin(\theta(k)) \\
\theta(k+1) &=& \frac{1}{4}\left( 1+ \sin (\theta(k)) \right) + \theta(k) .
\end{array}
\end{align}
The stability property of system~\eqref{eq:rec-sys-non-norm} can then be expressed in the new coordinates~\eqref{eq:pol-sys-non-norm} as stability with respect to the measurement function $\omega(r,\theta) = 1+\sin (\theta)$.

\section{Conclusions and Outlook} \label{sec:conclusions}
We have provided non-conservative small-gain conditions ensuring $\omega$ISS (ISS with respect to closed sets) for a network of discrete-time systems.
Toward this end, we first introduced a characterization of $\omega$ISS via finite-step $\omega$ISS Lyapunov functions.
This characterization was then used to derive suitable small-gain conditions proving $\omega$ISS of the network by constructing a suitable finite-step $\omega$ISS Lyapunov function for the overall system.
Necessity of these small-gain conditions was proven for a large class of systems.
We eventually presented several applications of our results including partial ISS, ISS for time-varying systems, synchronization problems, incremental stability, and distributed observers.

A relevant question arising from this paper regards how much of this theory carries over into the continuous-time domain.
Originally, the finite-step (or rather finite-``time'') approach proposed by~\cite{Aeyels.1998} 
was formulated for continuous-time systems, but without using the term \emph{finite-step}.
Accordingly, many results of the finite-step approach in discrete time can, in principle, be also derived similarly in continuous time.
For instance, compare the construction of Lyapunov functions via the finite-step approach in~\cite{Geiselhart.2014c} (discrete time) and~\cite{Doban.2016} (continuous time).
However, there are several challenges associated with this.
The first one is that there is no obvious notion of steps in continuous-time systems, so a finite-step Lyapunov function is not a very natural concept in the continuous-time domain.
Whereas in discrete time the computation of solutions boils down to iterating the dynamics map,
in continuous time to even compute the state of a node in a network an epsilon ahead, current states of all nodes in the network need to be known, unless the network is of a very specific structure.
Higher order derivatives of a single (or vector) Lyapunov function may seem like a canonical counterpart of finite-step Lyapunov functions, and indeed they have been considered in stability theory~\cite{lakshmikanthammatrosovsivasundaram1991-vector-lyapunov-functions-and-stability-analysis-of-nonlinear-systems}, but the aforementioned problem remains,
so that the effective decoupling of sparse yet cyclic networks
is not achieved via these functions.
Another approach could be to consider the sampled-data system corresponding to the continuous-time system~\cite{grune2002-asymptotic-behavior-of-dynamical-and-control-systems-under-perturbation-and-discretization}, and then to apply the theory presented here to the resulting discrete-time systems.
An obvious problem with that is, of course, that forward completeness of continuous-time systems, unlike the discrete-time case, is generally not a given, which adds further obstacles to this possible pathway for an extension.



\appendix

The global $\K$-boundedness can be seen as a worst-case estimate of the measurement update $\omega(g(\xi,\mu))$ compared to the measurement state $\omega(\xi)$ and the norm of the input $\abs{\mu}$.
In the proof of Theorem~\ref{thm:equivalence-of-Lyapunov-functions} we need such a worst-case estimate for a finite-step evolution.
The following lemma extends~\cite[Lemma~A.3]{Geiselhart.2016} to the case of a single measure.

\begin{lemma} \label{lem:evolutionKbound}
  Let system~\eqref{eq:1} be globally $\K$-bounded with respect to the measurement function $\omega$. Then for any $j \in \Zsp$ there exist $\K$-functions $\vartheta_j, \zeta_j$ such that for all $\xi \in \R^n$ and all $u(\cdot) \in \ell^\infty$, we have
  \begin{equation} \label{eq:evolutionKbound}
    \omega\big(x\big(j,\xi, u(\cdot)\big)\big) \leq \vartheta_j\big(\omega(\xi)\big) + \zeta_j(\norm{u}).
  \end{equation}
\end{lemma}

\begin{proof}
  The proof follows the lines of~\cite[Lemma~A.3]{Geiselhart.2016} and is thus only sketched. Here, we only give the induction step. Assume that the statement of Lemma~\ref{lem:evolutionKbound} holds true for some $j \in \Zsp$, i.e., there exist functions $\vartheta_j, \zeta_j \in \K$ satisfying~\eqref{eq:evolutionKbound}. Then, with $\kappa_1, \kappa_2\in \K$ coming from~Definition~\ref{def:gKb}, we have
\begin{align*}
\omega(x(j+1,\xi,u(\cdot)))  & =  \omega(g(x(j,\xi,u(\cdot))),u(j)) \\
& \leq \kappa_1(\omega(x(j,\xi,u(\cdot)))) + \kappa_2(\norm{u}) \\
& \leq  \kappa_1(\vartheta_j(\omega(\xi)) + \zeta_j(\norm{u})) + \kappa_2(\norm{u}) \\
& \leq  \kappa_1(2\vartheta_j(\omega(\xi))) + \kappa_1(2\zeta_j(\norm{u})) \\
& \qquad + \kappa_2(\norm{u}),
\end{align*}
which shows~\eqref{eq:evolutionKbound} for $j+1$ with functions $\vartheta_{j+1} (\cdot) = \kappa_1(2\vartheta_j(\cdot))$ and $\zeta_{j+1}(\cdot) = \kappa_1(2\zeta_j(\cdot)) + \kappa_2(\cdot)$.
\end{proof}

\vfill\mbox{}
\end{document}